\documentclass[letterpaper,10pt,conference,romanappendices]{ieeeconf}
\IEEEoverridecommandlockouts                              

\usepackage{graphics} 
\usepackage{graphicx}
\usepackage{epstopdf} 
\usepackage{times} 
\usepackage{amsmath} 
\usepackage{amssymb}  
\usepackage{theorem}
\usepackage{color,xspace}
\usepackage{hyperref}
\usepackage{graphicx}
\usepackage{subcaption}
\usepackage{cite}
\usepackage{cleveref}
\usepackage{bm}
\usepackage{bbm}
\usepackage{algorithm}
\usepackage{algpseudocode}
\usepackage{enumerate}

\input{mysymbol.sty}

\makeatletter

\newtheorem{theorem}{Theorem}
\newtheorem{lemma}{Lemma}	

\newtheorem{remark}{Remark}
	
\newtheorem{corollary}{Corollary}
	
\newtheorem{assumption}{Assumption}	
\newtheorem{problem}{Problem}

\def\bx{\mathbf{x}}
\def\bh{\mathbf{h}}
\def\bg{\mathbf{g}}

\def\by{\mathbf{y}}

\def\bbx{{\ensuremath{\mathbf x}}}

\newcommand*{\Scale}[2][4]{\scalebox{#1}{$#2$}}%
\newcommand{\oprocendsymbol}{\hbox{$\bullet$}}
\newcommand{\oprocend}{\relax\ifmmode\else\unskip\hfill\fi\oprocendsymbol}

\def \QED {$\square$}

\title{\LARGE Self-Triggered Time-Varying Convex Optimization\bf }
\author{Mahyar Fazlyab, Cameron Nowzari, George J. Pappas, Alejandro Ribeiro, Victor M. Preciado
	\thanks{The authors are with the Department of Electrical and Systems Engineering, University of Pennsylvania, Philadelphia, PA 19104, USA. Email: \{mahyarfa, cnowzari, pappasg, aribeiro, preciado\}@seas.upenn.edu.}}

\begin{document}

\maketitle
\thispagestyle{empty}
\pagestyle{empty}

\begin{abstract}
	In this paper, we propose a self-triggered algorithm to solve a
	class of convex optimization problems with time-varying objective functions. It is known that the trajectory of the optimal
	solution can be asymptotically tracked by a continuous-time state update law.
	Unfortunately, implementing this requires continuous evaluation
	of the gradient and the inverse Hessian of the objective function which is not amenable to
	digital implementation. Alternatively, we draw inspiration from
	self-triggered control to propose a strategy that autonomously
	adapts the times at which it makes computations about the objective
	function, yielding a piece-wise affine state update law. The algorithm
	does so by predicting the temporal evolution of the gradient using
	known upper bounds on higher order derivatives of the objective
	function. Our proposed method guarantees convergence to arbitrarily small neighborhood of the optimal trajectory in finite time and 
	without incurring Zeno behavior. We illustrate our framework with numerical simulations.
\end{abstract}

\begin{keywords}
	Time-varying optimization, self-triggered control, adaptive step size
\end{keywords}


\section{Introduction}

In this paper, we address a class of time-varying optimization problems where
the goal is to asymptotically track a unique, time-varying optimal trajectory given by
\begin{align}\label{eq:intro}
\bx^{\star}(t):=\underset{\bx \in \mathbb{R}^n}{\argmin} \quad f_{0}(\bx,t),\  t\in \mathbb{R}_{+}.
\end{align}
Problems of this form are generally referred to as time-varying optimization or parametric programming in the literature, and often arise in dynamical systems that involve an objective function or a set of constraints that have a dependence on time or a dynamic parameter, in general.
Particular examples include fast model predictive control using online convex optimization \cite{wang2010fast}, real time convex optimization in signal processing \cite{mattingley2010real},  distributed optimization of time-varying functions \cite{rahili2015distributed}, time-varying pose estimation \cite{baumann2004newton}, traffic engineering in computer networks \cite{su2009traffic}, neural network learning \cite{641452,716979}, and dynamic density coverage for mobile robots \cite{7050337}.

From an optimization perspective, a general framework for solving problem \eqref{eq:intro} is to sample the objective function at particular times of interest, and solve the corresponding sequence of stationary optimization problems by standard iterative algorithms such as gradient or Newton's methods. However, these algorithms clearly ignore the dynamic aspect of the problem which means they yield solutions with a final steady-state error whose magnitude is related to the time-varying aspects of the problem \cite{popkov2005gradient}.

From a dynamical systems perspective, one could perform time sensitivity analysis of the optimal solution to propose a continuous-time dynamical system whose state is asymptotically driven to the optimal solution \cite{716979,baumann2008newton}. The resulting dynamics is a combination of standard descent methods and a prediction term which tracks the drift in the optimal solution. For error-free tracking, however, we need to solve the dynamics continuously, implying that we need continuous access to the objective function and all of its derivatives that appear in the continuous-time dynamics. A natural solution to this is to implement the
continuous-time dynamics periodically. In a recent work \cite{simonetto2015class}, the authors proposed a periodic sampling strategy in which the objective function is periodically sampled with  a constant period $h>0$, and a single step of prediction along with multiple iterations of standard gradient or Newton's algorithm are combined to achieve an asymptotic error bound that depends on $h$ and the number of descent steps taken between the sampling times.

Instead, we are interested in utilizing self-triggered control strategies \cite{4303247,5411835,SA-CN-GJP:15,CN-JC:13-tac}  to adaptively determine when samples of the objective function are needed without sacrificing the convergence; see~\cite{WPMHH-KHJ-PT:12} for a survey.  From a dynamical systems perspective, this strategy plays a similar role as step size selection in stationary optimization, where a proper continuous-time dynamics $(\dot{\bx}(t)=-\nabla_{\bx}f(\bx(t))$ for instance) is discretized aperiodically using a backtracking line search method \cite{boyd2004convex}. In time-varying optimization, however, the line search method is no longer applicable as time and space become entangled. In this context, we can view our self-triggered sampling strategy as a way of adaptively choosing a proper step size in both time and space together.
There are similar works that propose event-triggered broadcasting strategies to solve static
distributed optimization problem~\cite{PW-MDL:09,MZ-CGC:10,DR-JC:14-cdc,SSK-JC-SM:15-auto},
but to the knowledge of the authors, no work has been reported on an aperiodic discretization of continuous time-varying optimization problems.

\medskip
\emph{Statement of contributions:} 
In this work we are interested in developing a real-time algorithm that can asymptotically
track the time-varying solution~$\bx^\star(t)$ to a time-varying optimization problem.
Our starting point is the availability of a continuous-time dynamics $\dot{\bx}(t) = \bh(\bx(t),t)$
such that the solutions to this satisfy $\| \bx(t) - \bx^\star(t)  \| \rightarrow 0$ as
$t \rightarrow \infty$. Then, we are interested in a real-time implementation such
that $\dot{\bx}(t)$ is to be updated at discrete instants of time and is held constant between updates. In contrast to standard methods that consider periodic samples, our contribution is the development of a self-triggered control strategy that autonomously determines how often $\dot{\bx}(t)$ should be updated. Intuitively, the self-triggered strategy determines
how long the current control input can be applied without negatively affecting the convergence.
Our algorithm guarantees that the state $\bx(t)$ can asymptotically track an arbitrarily small neighborhood around~$\bx^\star(t)$ while ensuring Zeno behavior is avoided. Simulations illustrate our results.

\emph{Notation} Let $\mathbb{R}$, $\mathbb{R}_{+}$, and $\mathbb{R}_{++}$ be the set of real, nonnegative, and strictly positive real numbers. $\mathbb{Z}_{+}$ and $\mathbb{Z}_{++}$ denote nonnegative and positive integers, respectively. $\mathbb{R}^n$ is the space of $n$-dimensional vectors and $\mathbb{S}^n$ is the space of $n$ by $n$ symmetric matrices. The one-norm and two-norm of $\bx \in \mathbb{R}^n$ is denoted by $\|\bx\|_1$ and $\|\bx\|_2$, respectively. The gradient of the function $f(\bx,t) \colon \mathbb{R}^n \times \mathbb{R}_{+} \to \mathbb{R}$ with respect to $\bx \in \mathbb{R}^n$ is denoted by $\nabla_{\bx} f(\bx,t) \colon \mathbb{R}^n \times \mathbb{R}_{+} \to \mathbb{R}^n$. The \emph{partial} derivatives of $\nabla_{\bx} f(\bx,t)$ with respect to $\bx$ and $t$ are denoted by $\nabla_{\bx\bx} f(\bx,t) \colon \mathbb{R}^n \times \mathbb{R}_{+} \to \mathbb{S}^n$ and $\nabla_{\bx t} f(\bx,t) \colon \mathbb{R}^n \times \mathbb{R}_{+} \to \mathbb{R}^n$, respectively. Higher order derivatives are also defined similarly.


\section{Preliminaries and Problem Statement} \label{Section: Problem Statement}
Let $\bbx \in \reals^n$ be a decision variable, $t \in \reals_{+}$ a time index, and $f \colon \reals^n \times \reals_{+} \to \reals$ a real-valued convex function taking values $f(\bbx, t)$. We interpret $f$ as a time-varying objective and consider the corresponding time-varying optimization problem in which we want to find the argument $\bbx^*(t)$ that minimizes the objective $f(\bbx, t)$ at time $t$,
Consider the function $f \colon \mathbb{R}^n \times \mathbb{R}_{+} \to \mathbb{R}$, and define the following optimization problem,
\begin{alignat}{2} 
	\bx^{\star}(t):=\underset{\bx \in \mathbb{R}^n}{\argmin} \quad f(\bx,t),  \label{eq: unconstrained_time_varying_problem}
\end{alignat} 
We impose the following assumption on $f(\bx,t)$.
\begin{assumption} \label{strong convexity}
	The objective function $f(\bx,t)$ is uniformly strongly convex in $\bx$, i.e., $f(\bx,t)$ satisfies $\nabla_{\bx\bx} f(\bx,t) \geq m \mathbf{I}_n$ for some $m>0$, and for all $t \in \mathbb{R}_{+}$.
\end{assumption}
By virtue of Assumption \ref{strong convexity}, $\bx^\star(t)$ is unique for each $t\in \mathbb{R}_{+}$ \cite{boyd2004convex}. \ifx If we apply standard iterative methods such as gradient or Newton's method for solving \eqref{eq: unconstrained_time_varying_problem}, the generated trajectory would lie within a distance of the optimal trajectory whose magnitude depends on the temporal variations of the objective function. More precisely, these iterative algorithms fail to achieve error-free tracking. \fi
The optimal trajectory {$\bx^{\star}(t)$} is then implicitly characterized by $\nabla_{\bx} f(\bx^{\star}(t),t)= 0$ for all $t \in \mathbb{R}_{+}$. Using the chain rule to differentiate this identity with respect to time yields 
\begin{alignat}{2}
	\Scale[0.96]{\dfrac{d}{dt} \nabla_{\bx} f(\bx^{\star}(t),t)= \nabla_{\bx\bx} f(\bx^{\star}(t),t) \dfrac{d}{dt}\bx^{\star}(t)+\nabla_{\bx t} f(\bx^{\star}(t),t)}.
\end{alignat}
%
Since the left-hand side of the above equation is identically zero for all $t \in \mathbb{R}_{+}$, it follows that the optimal point obeys the dynamics
\begin{align} \label{eq: optimal_point_velocity}
	\dfrac{d}{dt}\bx^{\star}(t) =-\nabla_{\bx\bx}^{-1} f(\bx^{\star}(t),t)\nabla_{\bx t} f(\bx^{\star}(t),t).
\end{align}
The above dynamics suggests that the optimizer needs to follow the minimizer with the same dynamics, in addition to taking a descent direction in order to decrease the suboptimality. Choosing Newton's method as a descent direction yields the following continuous-time dynamics,
\begin{align} \label{eq: time_varying_newton}
	\dfrac{d}{dt}\bx(t)=\bh(\bx(t),t),
\end{align}
where the vector field $\bh: \mathbb{R}^n \times \mathbb{R}_{+} \to \mathbb{R}^n$ is given by,
\begin{align} \label{eq: descent_vector_field}
\bh(\bx,t)=-{\nabla_{\bx\bx}^{-1} f(\bx,t)}[\alpha\nabla_{\bx} f(\bx,t)+\nabla_{\bx t} f(\bx,t)],
\end{align}
Here $\alpha \in \mathbb{R}_{++}$ is arbitrary. An appropriate Lyapunov function for \eqref{eq: time_varying_newton} is
\begin{align} \label{eq: Lyapunov function}
V(\bx,t):= \dfrac{1}{2} \|\nabla_{\bx} f(\bx,t)\|_2^2,
\end{align}
which is zero along the optimal path, i.e., $V(\bx^\star(t),t)=0,\ t\geq 0$. It can be verified that under the continuous-time dynamics \eqref{eq: time_varying_newton}, the Lyapunov function evaluated at $(\bx(t),t)$ satisfies the ODE 
\begin{align} \label{eq: Lyapunov ODE}
\dot{V}(\bx(t),t) = -2\alpha V(\bx(t),t).
\end{align}
Solving the latter ODE yields the closed form solution $V(\bx(t),t)=V(\bx(t_0), t_0) \exp(-2\alpha (t-t_0))$, where $\bx(t_0) \in \mathbb{R}^n$ is the initial point, and $t_0 \in \mathbb{R}_{+}$ is the initial time assumed to be zero. This implies that exponential convergence of $\bx(t)$ to $\bx^\star(t)$ requires continuous evaluation of the gradient and the inverse Hessian of the objective function, according to \eqref{eq: time_varying_newton} and \eqref{eq: descent_vector_field}, which is computationally expensive and is not amenable to digital implementation. Instead, we can use a simple Euler method to discretize \eqref{eq: time_varying_newton}. More precisely, suppose we use a sequence of periodic sampling times $\{t_k\}_{k\in \mathbb{Z}_{++}}$ with period $\tau>0$, i.e., $t_{k+1}-t_k=\tau$ for any $k\in \mathbb{Z}_{+}$ to arrive at the following piece-wise affine state update law,
\begin{align} \label{eq: discretized dynamics}
\dfrac{d}{dt}\hat{\bx}(t)=\bh(\hat{\bx}(t_k),t_k) ,\ t \in [t_k,t_{k+1}).
\end{align}
where $\hat{\bx}(t_k)$ is the estimate of $\bx(t_k)$ obtained from the ideal dynamics \eqref{eq: time_varying_newton}. Now if the vector field $\bh(\bx,t)$ is uniformly Lipschitz in $\bx$, i.e., for all $\bx,\by \in \mathbb{R}^n$, we have the inequality $\|\bh(\bx,t)-\bh(\by,t)\| \leq L\|\bx-\by\|$ for some $L>0$, and that the initial condition satisfies $\hat{\bx}(t_0) =\bx(t_0)$, the discretization error at time $t_k$ would satisfy the bound \cite{iserles2009first}
\begin{align} \label{eq: global truncation error}
\|\hat{\bx}(t_k) - \bx(t_k)\|_2 \leq \dfrac{c\tau}{L}\left[(1+\tau L)^k-1\right], \quad k\in \mathbb{Z}_{+}.
\end{align}
The above inequality implies that $\|\hat{\bx}(t_k) - \bx(t_k)\| \to 0$ as $\tau\to 0 $, i.e., we can control only the order of magnitude of the discretization error by $\tau$. Moreover, this upper bound is very crude and is of no use for quantifying the suboptimality of $\hat{\bx}(t_k)$ \cite{iserles2009first}. Motivated by this observation, we are interested in a sampling strategy that autonomously adapts the times at which it makes computations about the objective function in order to control the suboptimality. We formalize the problem next.
\begin{problem} \label{problem}
Given the dynamics~\eqref{eq: discretized dynamics}, find a strategy that determines the least frequent sequence of sampling times $\{t_k\}_{k\in \mathbb{Z}_{++}}$ such that:
\begin{enumerate}[(i)]
	\item for each $k\in\mathbb{Z}_{+}$, $t_{k+1}$ is determined without having access to the objective function for $t> t_k$,
	\item $\hat{\bbx}(t)$ converges to any neighborhood of the optimal trajectory after a finite number of samples, and remains there forever, and
	\item $t_{k+1}-t_{k} > c >0$ for some $c \in \mathbb{R}_{++}$ and all $k \in \mathbb{Z}_{++}$.
\end{enumerate}
\end{problem}
The first property guarantees that the proposed method is completely online. The second property enables the optimizer to arbitrarily bound the discretization error. The last property ensures Zeno behavior is avoided. In order to develop the main results, we make the following Assumption about the objective function.
\begin{assumption}\label{bounds on second-order derivatives} 
	The objective function $f(\bx,t)$ is thrice continuously differentiable. Furthermore,
	\begin{itemize}
		\item[\emph{(i)}] The second-order derivatives are bounded by known positive constants, i.e.,
		$$
		\|\nabla_{\bx\bx} f(\bx,t)\|_2 \leq C_{xx}, \quad  \|\nabla_{\bx t} f(\bx,t)\|_2 \leq C_{xt}.
		$$
		\item[\emph{(ii)}] The third-order derivatives are bounded by known positive constants, i.e.,
		$$
		\|\nabla_{\bx\bx t} f(\bx,t)\|_2 \leq C_{xxt}, \quad \|\nabla_{\bx t t} f(\bx,t)\|_2 \leq C_{xtt}.
		$$
		\begin{align*}
		\|\nabla_{\bx\bx\bx_i} f(\bx,t)\|_2 \leq C_{xxx},\ i \in [n].
		\end{align*}
		where $\nabla_{\bx\bx\bx_i} f(\bx,t) := \dfrac{\partial}{\partial \bx_i} \nabla_{\bx\bx} f(\bx,t)$.
	\end{itemize}
\end{assumption}
The first bound in Assumption \ref{bounds on second-order derivatives}-\emph{(i)} and the last bound in Assumption \ref{bounds on second-order derivatives}-\emph{(ii)} are equivalent to uniform Lipschitz continuity of the gradient and the Hessian in $\bx$, respectively. Both assumptions are standard assumptions in second-order methods \cite{boyd2004convex}.  All other bounds are related to the time-varying aspect of the objective function and bound the rate at which the gradient and  Hessian functions vary with time. Notice that except for the bound $\|\nabla_{\bx t t} f(\bx,t)\|_2 \leq C_{xtt}$, all the other bounds are required for $\bh(\bx,t)$ to be uniformly Lipschitz \cite{simonetto2015class}.


\section{Self-Triggered Strategy}
In this section, we design a self-triggered sampling strategy that meets the desired specifications defined in Problem \ref{problem}. Consider the discrete implementation of the ideal dynamics \eqref{eq: time_varying_newton} at a sequence of times $\{ t_k \}_{k \in \mathbb{Z}_{+}}$ that is to be determined,
\begin{align} \label{eq: aperiodic sampling }
\dot{{\bx}}(t) = \bh(\bx(t_k),t_k),\  t_k \leq t < t_{k+1}.
\end{align}
Recalling the Lyapunov function \eqref{eq: Lyapunov function}, the instantaneous derivatives of $V(\bx,t)$ at the discrete sampling times~$\{ t_k \}$
are precisely
\begin{align} \label{eq: V_dot at sampling times}
\dot{V}(\bx(t_k),t_k) = -2\alpha V(\bx(t_k),t_k),\ k \in \mathbb{Z}_{+}.
\end{align}
In other words, the property~\eqref{eq: Lyapunov ODE} that holds at all times in the continuous-time framework is now only preserved at discrete sampling times. This means in general there is no guarantee that $\dot{V}(t)$ remains negative between sampling times $t \in (t_k, t_{k+1})$, as the optimizer is no longer updating its dynamics during this time interval. Given these observations, we need to predict, without having access to the objective function or its derivatives for $t > t_k$,  the earliest time after $t_k$ at which the Lyapunov function could possibly increase, and update the state dynamics at that time, denoted by $t_{k+1}$. Consequently, we desire a \emph{tight upper bound} on $\dot{V}(t)=\dot{V}(\bx(t),t)$ so that we are taking samples as conservatively as possible. Mathematically speaking, for each $t\geq t_k$, we can characterize the upper bound as follows,
\begin{align} \label{eq: triggering function def}
\phi_k(t) = &\underset{\mathcal{F}}{{\mbox{sup}}} \ \{\dot{V}(\bx(t),t) \colon \dot{\bx}(t)=\dot{\bx}(t_k),\ t\geq t_k\}.
\end{align}
where $\mathcal{F}$ is the class of all strongly convex objective functions $f' \colon \mathbb{R}^n \times \mathbb{R}_{+} \to \mathbb{R}$ such that
\begin{enumerate} \label{eq: F properties}
	\item $\nabla_{\bx} f'(\bx(t_k),t_k)=\nabla_{\bx} f(\bx(t_k),t_k)$,
	\item $\nabla_{\bx t} f'(\bx(t_k),t_k)=\nabla_{\bx t} f(\bx(t_k),t_k)$,
	\item $\nabla_{\bx\bx} f'(\bx(t_k),t_k)=\nabla_{\bx\bx} f(\bx(t_k),t_k)$,
	\item $f'(\bx,t)$ satisfies Assumption \ref{bounds on second-order derivatives}.\\
\end{enumerate}
In words, $\mathcal{F}$ is the set of all possible objective functions that agree with $f(\bx,t)$ and it's first and second-order derivatives at $(\bx(t_k),t_k)$, and satisfy the bounds in Assumption \ref{bounds on second-order derivatives}. Intuitively, the set $\mathcal{F}$ formalizes, in a functional way, the fact that we find $\phi_k(t)$ without having access to the objective function for $t>t_k$.
%
The above definition implies that $\dot{V}(t_k) \leq \phi_k(t)$.  In particular, we have that $\dot{V}(t_k)=\phi_k(t_k)=-2\alpha {V}(t_k) <0$ by \eqref{eq: triggering function def} and \eqref{eq: V_dot at sampling times}. Once $\phi_k(t)$ is characterized at time $t_k$ as a function of $t$, the next sampling time is set as the first time instant at which $\phi_k(t)$ crosses zero, i.e.,
\begin{align} \label{eq: next sampling time}
t_{k+1}={\phi^{-1}_k}(0),\ k\in \mathbb{Z}_{+}.
\end{align}
where $\phi_k^{-1}(.)$ is the inverse of the map $\phi_k(.)$. This choice ensures that $\dot{V}(t)\leq \phi_k(t) < \phi_k(t_{k+1}) =0$ for $t\in [t_k,t_{k+1})$. With this policy, the evaluated Lyapunov function $V(t)$ becomes a piece-wise continuously differentiable monotonically decreasing function of $t$ with discontinuous derivatives at the sampling times. We can view $\phi_k(t)$ as a triggering function which triggers the optimizer to sample when the event $\phi_k(t')=0$ occurs for some $t'>t_k$. This concept is illustrated in Figure \ref{fig: self-triggered concept}. In the next Subsection, we characterize $\phi_k(t)$ in closed-form.
\begin{figure}
	\centering
	\includegraphics[width=0.45\textwidth]{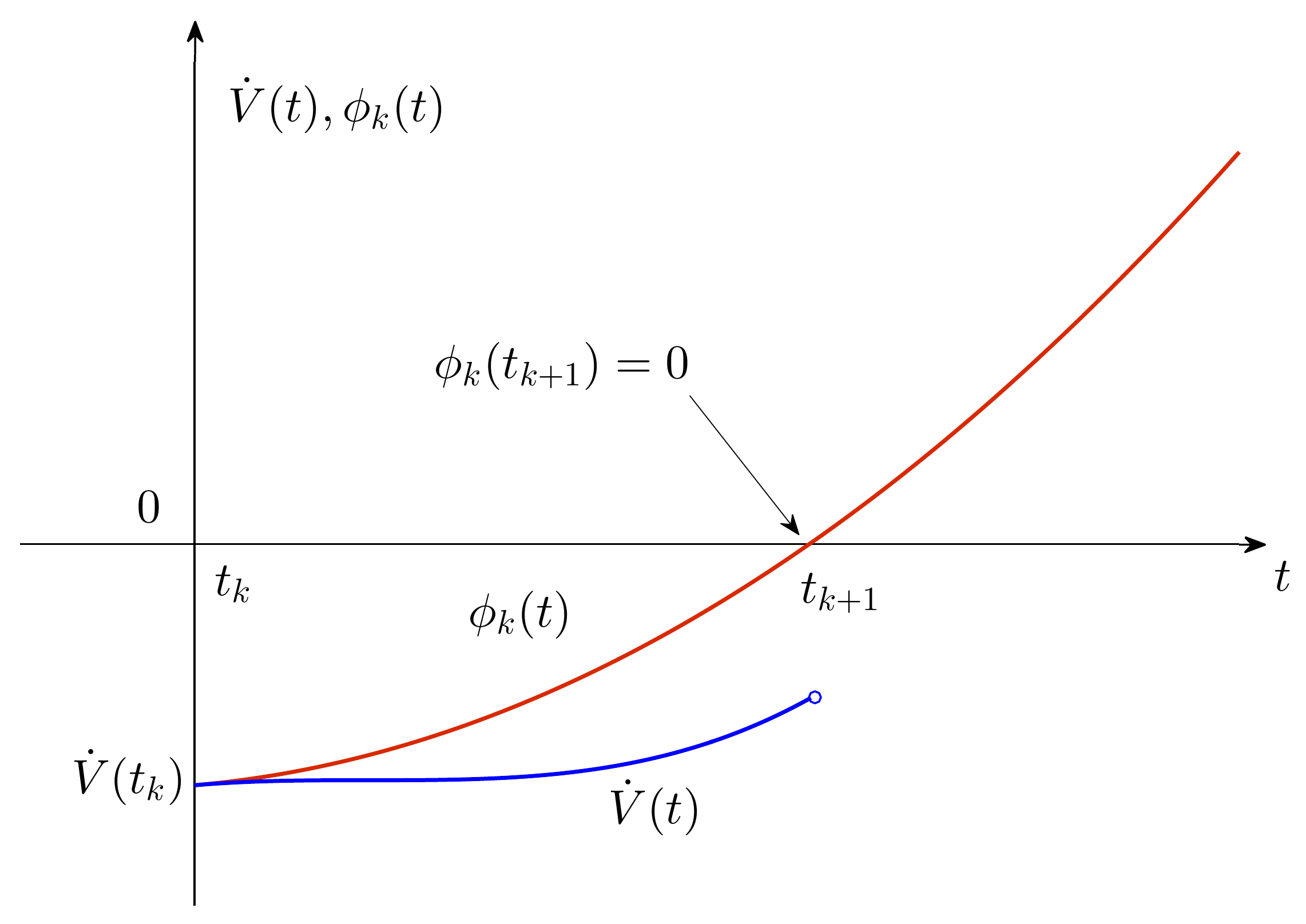}
	\caption{Concept of the self-triggered strategy. The triggering function $\phi_k(t)$ is a tight upper bound on $\dot{V}(t)$, and the optimizer is triggered to sample when the event  $\phi_k(t')=0$ occurs for some $t'>t_k$.}
	\label{fig: self-triggered concept}
\end{figure}

\subsection{Triggering Function} \label{subsection: triggering function}

\begin{lemma}[Second-order self-triggered strategy] \label{lem: trigerring function}
	Let $k \in \mathbb{Z}_{+}$. Then, given the bounds $\{C_{xx},C_{xt},C_{xxx},C_{xxt},C_{xtt}\}$ in Assumption \ref{bounds on second-order derivatives}, the triggering function $\phi_k(t)$ on the time interval $t_k \leq t \leq t_{k+1}$ is given by the following second-order polynomial,
	\begin{align} \label{eq: second order triggering function}
	\phi_k(t) &= \dfrac{1}{2} a_kb_k (t-t_k)^2 + (a_k^2+b_k \sqrt{2V(t_k)}) (t-t_k) \nonumber \\ &-2\alpha V(t_k).
	\end{align}
	where $a_k,\ b_k >0$ are computed as
	\begin{align} \label{eq: 2nd order coefficients}
	a_k &= \left(C_{xx} \|\dot{\bx}(t_k)\|_2 + C_{xt} \right), \nonumber \\
	b_k &= \left(C_{xxx} \|\dot{\bx}(t_k)\|_1 + 2C_{xxt} \right) \|\dot{\bx}(t_k)\|_2 + C_{xtt}.
	\end{align}
	and $\dot{\bx}(t_k)$ is computed according to \eqref{eq: aperiodic sampling }.
	\ifx
	\[
	\Scale[0.9]{\dot{\bx}(t_k) = -{\nabla_{\bx\bx}^{-1} f_0(\bx(t_k),t_k)}[\alpha\nabla_{\bx} f_0(\bx(t_k),t_k)+\nabla_{\bx t} f_0(\bx(t_k),t_k)]}.
	\]
	\fi
\end{lemma}
\begin{proof}
	See Appendix \ref{proof:em: trigerring function}.
	\end{proof}

\ifx
\begin{table*}[t]
	\centering
	\begin{tabular}{ c |l }
		\hline
		\\ Sampling Strategy & Self-Triggering Function \\ \\
		\hline \\
			\multirow{3}{*}{Second Order Self Triggered Optimizer} & $\phi_k(t)=\frac{1}{2} a_kb_k (t-t_k)^2 + (a_k^2+b_k \sqrt{2V(t_k)}) (t-t_k) -2\alpha V(t_k), \quad t \geq t_k$ \\ \\
			& $a_k=\left(C_{xx} \|\dot{\bx}(t_k)\|_2 + C_{xt} \right)$ \\ \\
			& $b_k= \left(C_{xxx} \|\dot{\bx}(t_k)\|_1 + 2C_{xxt} \right) \|\dot{\bx}(t_k)\|_2 + C_{xtt}$ \\ \\
			\hline
		\\ Third Order Self Triggered Optimizer & $\psi_k(t)=\frac{1}{2} b_k^2 (t-t_k)^3 + \frac{3}{2} \alpha \sqrt{2V(t_k)} b_k(t-t_k)^2 + \left(\sqrt{2V(t_k)}b_k + 2\alpha^2 V(t_k) \right)(t-t_k)-2\alpha V(t_k),  \quad t \geq t_k$ \\ \\
		& $b_k= \left(C_{xxx} \|\dot{\bx}(t_k)\|_1 + 2C_{xxt} \right) \|\dot{\bx}(t_k)\|_2 + C_{xtt}$ \\ \\
		\hline
	\end{tabular}
	
	\caption{Blabla}
	\label{tab:1}
\end{table*}
\fi
A careful investigation into the Proof of Lemma \ref{lem: trigerring function} reveals the fact that $\phi_k(t)$ can still be characterized without having access to the bounds in Assumption-\emph{(i)}, i.e., we assume known bounds on the third-order derivatives only. By virtue of this relaxation, our strategy prescribes less conservative sampling times. We will see in the next lemma that the triggering function, in this case, is a third-order polynomial.

\begin{lemma}[Third-order self-triggered strategy] \label{lem: trigerring function third order}
	Let $k \in \mathbb{Z}_{+}$. Then, given the bounds $\{C_{xxx},C_{xxt},C_{xtt}\}$ in Assumption \ref{bounds on second-order derivatives}, the triggering function $\phi_k(t)$ on the time interval $t_k \leq t \leq t_{k+1}$ is given by
	\begin{align} \label{eq: third order triggering function}
	\phi_k(t):=&\frac{1}{2} b_k^2 (t-t_k)^3 + \frac{3}{2} \alpha \sqrt{2V(t_k)} b_k(t-t_k)^2 \nonumber \\ &+ \left(\sqrt{2V(t_k)}b_k + 2\alpha^2 V(t_k) \right)(t-t_k)-2\alpha V(t_k).
	\end{align}
	with $b_k$ previously defined in \eqref{eq: 2nd order coefficients}.
\end{lemma}	
\begin{proof}
	See Appendix \ref{proof:lem: trigerring function third order}.
\end{proof}
\begin{remark} \label{rem: sampling times}
	{\rm It can be observed from \eqref{eq: second order triggering function} and \eqref{eq: third order triggering function} that $\phi_k(t)$ is fully characterized at time $t_k$ without having access to the objective function for $t>t_k$. In this context, the self-triggered strategy is online, implying the property \emph{(i)} in Problem \ref{problem}. Moreover, $\phi_k(t)$ has a unique root on the interval $(t_k,\infty)$ when $V(t_k)> 0$, implying that the sampling time $t_{k+1}=\phi^{-1}_k(0)$ is well-defined and the step size satisfies $t_{k+1} - t_k > 0$ for all $k$.	\oprocend}
\end{remark}
\subsection{Asymptotic Convergence}
The triggering functions developed in the previous lemmas have the following properties by construction:
\begin{enumerate}
	\item[(a)] $\phi_k(t)$ is convex in and strictly increasing on $t_k \leq t \leq t_{k+1}$.
	\item[(b)] $\dot{V}(t) \leq \phi_k(t)<0$ on  $t_k \leq t < t_{k+1}$.
	\item[(c)]$\phi_k(t_k)=\dot{V}(t_k)=-2\alpha V(t_k)$.
	\item[(d)] $\phi_k(t_{k+1})=0$.
\end{enumerate}
We establish in the next theorem that the above properties are enough to secure asymptotic monotone convergence of the Lyapunov function to zero.
\begin{theorem} \label{thm: convergence of second order}
	Let $\{t_k\}_{k \in \mathbb{Z}_{++}}$ be the sequence of sampling times generated according to \eqref{eq: next sampling time}, where $\phi_k(t)$ is defined in \eqref{eq: second order triggering function} or \eqref{eq: third order triggering function}. Then, for any $k \in \mathbb{Z}_{+}$ the Lyapunov function satisfies $V(t_{k+1})<V(t_k)$, and that $\lim\limits_{k \to \infty} V(t_k)=0$.
\end{theorem}
\begin{proof}
See Appendix \ref{proof:thm: convergence of second order}.
\end{proof}
\begin{remark}[Role of $\alpha$] \label{rem: role alpha}
	{\rm In the proof of Theorem \ref{thm: convergence of second order}, we showed that the Lyapunov function at the sampling times satisfies the inequality
	\[
	V(t_{k+1}) - V(t_{k}) \leq -\alpha V(t_k) (t_{k+1}-t_k).
	\]
	Combining this inequality with the trivial inequality $-V(t_k) \leq V(t_{k+1})-V(t_k)$ lets us conclude that for all $k \in \mathbb{Z}_+$, the step sizes are bounded as $t_{k+1}-t_k \leq \alpha^{-1}$. Therefore, increasing $\alpha$ will reduce the step sizes such that the effective step size $\alpha(t_{k+1}-t_k)$ is bounded by one. This observation is consistent with backtracking line search method in stationary optimization in which the step sizes are bounded by one.}
\end{remark}
We have the following corollary as an immediate consequence of Theorem \ref{thm: convergence of second order}.
\begin{corollary} \label{cor: finite time convergence}
		Let $\{t_k\}_{k \in \mathbb{Z}_{++}}$ be the sequence of sampling times generated according to \eqref{eq: next sampling time}, where $\phi_k(t)$ is defined in \eqref{eq: second order triggering function} or \eqref{eq: third order triggering function}. Then, for any $\epsilon>0$, there exist a finite positive integer $k'(\epsilon) \in \mathbb{Z}_{+}$ such that $V(t_{k'(\epsilon)})< \epsilon$.
\end{corollary}

Next, we discuss the implementation issues of the self-triggered strategy proposed above.
\subsection{Implementation}

It can be seen from Theorem \ref{thm: convergence of second order} and the expression of $\phi_k(t)$ in \eqref{eq: second order triggering function} or \eqref{eq: third order triggering function} that as $k \to \infty$, $V(t_k) \to 0$, and therefore $t_{k+1} - t_k \rightarrow 0$, i.e., the step sizes vanish asymptotically. This might cause Zeno behavior, i.e., the possibility for infinitely many samples over a finite interval of time.  To avoid this possibility, we need to modify the algorithm to ensure that the step sizes are lower bounded by a positive constant all the time; a stronger property than no Zeno behavior. For this purpose, we implement the algorithm in two phases: In the first phase, we use the sampling strategy developed in Subsection \ref{subsection: triggering function} until the state $\bx(t)$ reaches within a pre-specified neighborhood around $\bx^\star(t)$. In the second phase, we switch the triggering strategy so as to merely maintain $\bx(t)$ in that neighborhood forever. More specifically, for the sequence of sampling times $\{t_k\}_{k\in \mathbb{Z}_{+}}$ and any $\epsilon>0$, define 
$$k'(\epsilon)=\min \{k\in \mathbb{Z}_{+} \colon V(t_k) \leq \epsilon\}.$$
In words, $t_{k'(\epsilon)}$ is the first sampling time at which the Lyapunov function is below the threshold $\epsilon$. By Corollary \ref{cor: finite time convergence}, $k'(\epsilon)$ is finite. Now for $t \geq t_{k'(\epsilon)}$, we propose another self-triggered sampling strategy such that the Lyapunov function satisfies $V(t) \leq \epsilon$ for all $t\geq t_{k'(\epsilon)}$. Recalling the inequality $\dot{V}(t) \leq {\phi}_k(t)$, we can obtain an upper bound for $V(t)$ as follows,
\begin{align}
V(t) \leq V(t_k) + \int_{t_k}^{t} \phi_k(\sigma)d\sigma,\ t\geq t_k.
\end{align}
The right-hand side is a polynomial in $t$ which can be fully characterized at  $t_k$. Now for $k \geq k'(\epsilon)$, we set the next sampling time $t_{k+1}$ as the first time instant after $t_k$ at which the upper bound function in the right-hand side crosses $\epsilon$, i.e., we select $t_{k+1}$ according to the following rule,
\begin{align} \label{eq: sampling time phase two}
t_{k+1} = \psi_k^{-1}(\epsilon),\ k \in \mathbb{Z}_{+}.
\end{align}
where the new triggering function is defined as
\begin{align} \label{eq: psi function}
\psi_k(t):=V(t_k) + \int_{t_k}^{t} \phi_k(\sigma)d\sigma,\ t\geq t_k.
\end{align}
This policy guarantees that $V(t) \leq \psi_k(t) \leq \psi_k(t_{k+1}) = \epsilon$ for all $k > k'(\epsilon)$. As a result, by virtue of strong convexity \cite{boyd2004convex}, i.e.,  the inequality $\|\bx(t)-\bx^\star(t)\|_2 \leq 2/m \|\nabla_{\bx} f(\bx(t),t)\|_2$, and recalling \eqref{eq: Lyapunov function}, the following bound
\begin{align}
\|\bx(t)-\bx^\star(t)\|_2 \leq \dfrac{2 ({2\epsilon})^{\frac{1}{2}}}{m}.
\end{align}
will hold for all $t \geq t_{k'(\epsilon)}$. 
\begin{remark}
	{\rm According to \eqref{eq: sampling time phase two}, the sampling time is obtained by solving the equation $\psi_k(t)=\epsilon$ for $t>t_k$.  By \eqref{eq: psi function}, it can be verified that $t_{k+1}>t_k$ for $0 \leq V(t_k) < \epsilon$. In the limiting case $V(t_k)=\epsilon$, there are two solutions to $\psi_k(t)=\epsilon$; the first one is $t_{k+1}=t_k$ which implies zero step size and is ignored; the other one is strictly greater that $t_k$, and is selected as the next sampling time.}
\end{remark}

We summarize the proposed implementation in Table \ref{tab: Second Order Self-Triggered Optimizer}, where we use the notation $\bx_k:=\bx(t_k)$ and $\dot{\bx}_k:=\dot{\bx}(t_k)$.

\begin{algorithm}[h]
	{\footnotesize \vspace*{1ex}
		Second-order self-triggered strategy \\
		\textbf{Given}: $C_{xx},\ C_{xt},\ C_{xxx},\ C_{xxt},\ C_{xtt},\ \alpha,\ t_0,\ t_f,\ \bx(t_0),\ \epsilon.$
		\\ \\
		Third-order self-triggered strategy \\
		\textbf{Given}: $C_{xxx},\ C_{xxt},\ C_{xtt},\ \alpha,\ t_0,\ t_f,\ \bx(t_0),\ \epsilon.$ \\
		\begin{algorithmic}[1]
			\State Initialization: Set $k=0$, and $\bx_0 = \bx(t_0)$.
			\While {$t_k<t_f$}
			\State Compute $$\dot{\bx}_k =-{\nabla_{\bx\bx}^{-1} f_0(\bx_k,t_k)}[\alpha\nabla_{\bx} f_0(\bx_k,t_k)+\nabla_{\bx t} f_0(\bx_k,t_k)].$$
			\If {$\|\nabla_{\bx} f_0(\bx_k,t_k)\|_2 \geq (2\epsilon)^{\frac{1}{2}}$} \\
			\State Compute $t_{k+1}=\phi^{-1}_k(0)$  from \eqref{eq: second order triggering function} (or \eqref{eq: third order triggering function}).
			\Else 
			\State Compute $t_{k+1}=\psi^{-1}_k(\epsilon)$ from \eqref{eq: psi function}.
			\EndIf
			\State Update $\quad \bx_{k+1} = \bx_k + \dot{\bx}_k(t_{k+1}-t_k)$.
			\State Update $\quad k = k+1$.
			\EndWhile
		\end{algorithmic}}
		\caption{\hspace*{-.5ex}: Self-triggered optimizer} \label{tab: Second Order Self-Triggered Optimizer}
	\end{algorithm}
We close this section by the following theorem which accomplishes the main goals defined in Problem \ref{problem}.
\begin{theorem} \label{lem: finite time convergence}
	Let $\{t_k\}_{k \in \mathbb{Z}_{+}}$ be the sequence of sampling times generated by Algorithm \ref*{tab: Second Order Self-Triggered Optimizer}. Then, for any $\epsilon>0$, there exists a nonnegative integer $m \in \mathbb{Z}_{+}$ such that: (i) $V(t_m)<\epsilon$ for all $t \geq t_m$; and (ii) $t_{k+1}-t_{k} > \tau(\epsilon)$ for all $k \in \mathbb{Z}_{+}$ and some $\tau(\epsilon)>0$.
\end{theorem}
\begin{proof}
	The first statement follows directly from Corollary \ref{cor: finite time convergence}. For the proof of the second statement, see Appendix \ref{proof:thm: No Zeno}.
\end{proof}

\section{Simulation}
In this section, we perform numerical experiments to illustrate our results. For simplicity in our exposition, we consider the following convex problem in one-dimensional space
\begin{align*}
x^\star(t) = \arg\min \ \frac{1}{2} (x-\cos(\omega t))^2 + \frac{k}{2} \cos^2(2\omega t) \exp({\mu x^2}).
\end{align*} 
where $x \in \mathbb{R}$, $t \in \mathbb{R}_{+}$, $\omega=\pi/5$, $k=2$, and $\mu=1/2$. For these numerical values, we have that $C_{xxx}=3.7212$, $C_{xxt}=2.6924$, and $C_{xtt}=6.9369$. We solve this problem for the time interval $t\in [0,7]$ via Algorithm \ref*{tab: Second Order Self-Triggered Optimizer} using the triggering function \eqref{eq: third order triggering function}, and setting $\alpha=5$ and $\epsilon=0.01$. The total number of updates are $N=108$, with the step sizes having a mean value of $\bar{h}=0.0662$ and standard deviation $\sigma = 0.0501$. For comparison, we also solve the optimization problem by a more standard periodic implementation. We plot all the solutions $\bx(t)$ in Figure~\ref{fig: x(t)} along with the $\log_{e}$ of the total number of samples required in each execution. It can be observed that small sampling periods, e.g., $h = 0.0001,0.001,0.01$, yield a convergence performance similar to the self-triggered strategy, but uses a far higher number of updates. On the other hand, larger sampling periods, e.g., $h=0.1,0.2,0.3$, result in comparable number of samples as the self-triggering strategy at the expense of slower convergence. It should also be noted that we do not know a priori what sampling period yields good convergence results with a reasonable number of requires samples; however, the self-triggered strategy is capable of automatically tuning the step sizes to yield good performance while utilizing a much smaller number of samples.
\begin{figure}
	\centering
	\includegraphics[width=0.45\textwidth]{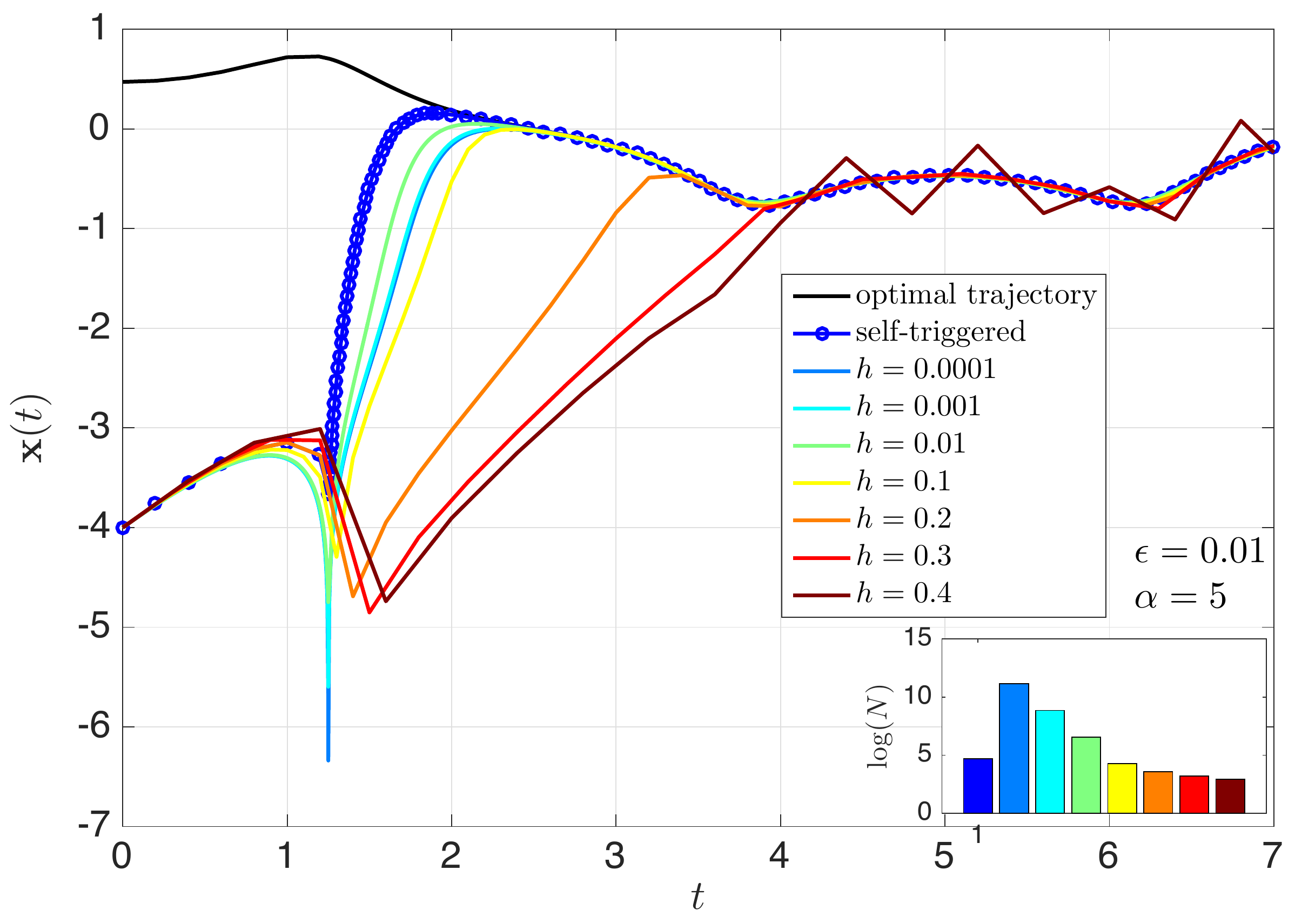}
	\caption{Plot of $\bx(t)$ against $t$ for the self-triggered strategy in Algorithm \ref{tab: Second Order Self-Triggered Optimizer}, and for periodic discretization with various sampling periods.}
	\label{fig: x(t)}
\end{figure}

\emph{Effect of $\epsilon$:} Next, we study the effect of the design parameter $\epsilon$ on the number of samples and the convergence performance of the self-triggered strategy. More specifically, we run Algorithm \ref{tab: Second Order Self-Triggered Optimizer} with all the parameters as before, and with different values of $\epsilon$. Figure \ref{fig: effect of eps} shows the resulting trajectories for various values of $\epsilon$. It is observed that $\epsilon$ does not change neither the transient convergence phase, but rather affects the steady state tracking phase. Moreover, the number of samples are almost unaffected by changing $\epsilon$.

\emph{Effect of $\alpha$:} Finally, we study the performance of the self-triggered strategy as $\alpha$ changes. Intuitively, higher values of $\alpha$ puts more weight on the descent part of the dynamics ($\nabla_{\bx\bx}^{-1} f(\bx,t)\nabla_{\bx} f(\bx,t)$) than the tracking part ($\nabla_{\bx\bx}^{-1} f(\bx,t)\nabla_{\bx t} f(\bx,t)$), according to \eqref{eq: time_varying_newton}. Hence, we expect that we get more rapid convergence to the $\epsilon$-neighborhood of the optimal trajectory by increasing $\alpha$. Figure \ref{fig: effect of alpha} illustrates the resulting trajectories for different values of $\alpha$. As expected, as we increase $\alpha$, the trajectory converges faster to the optimal trajectory. The number of samples, however, are not affected by $\alpha$. This observation is in agreement with Remark \ref*{rem: role alpha}, where we showed that the \emph{effective} step sizes $\alpha(t_{k+1}-t_k)$ are between zero and one, independent of $\alpha$. In the limiting case $\alpha \to \infty$, the step sizes get arbitrarily small which is not desirable.

\begin{figure}
	\centering
	\includegraphics[width=0.45\textwidth]{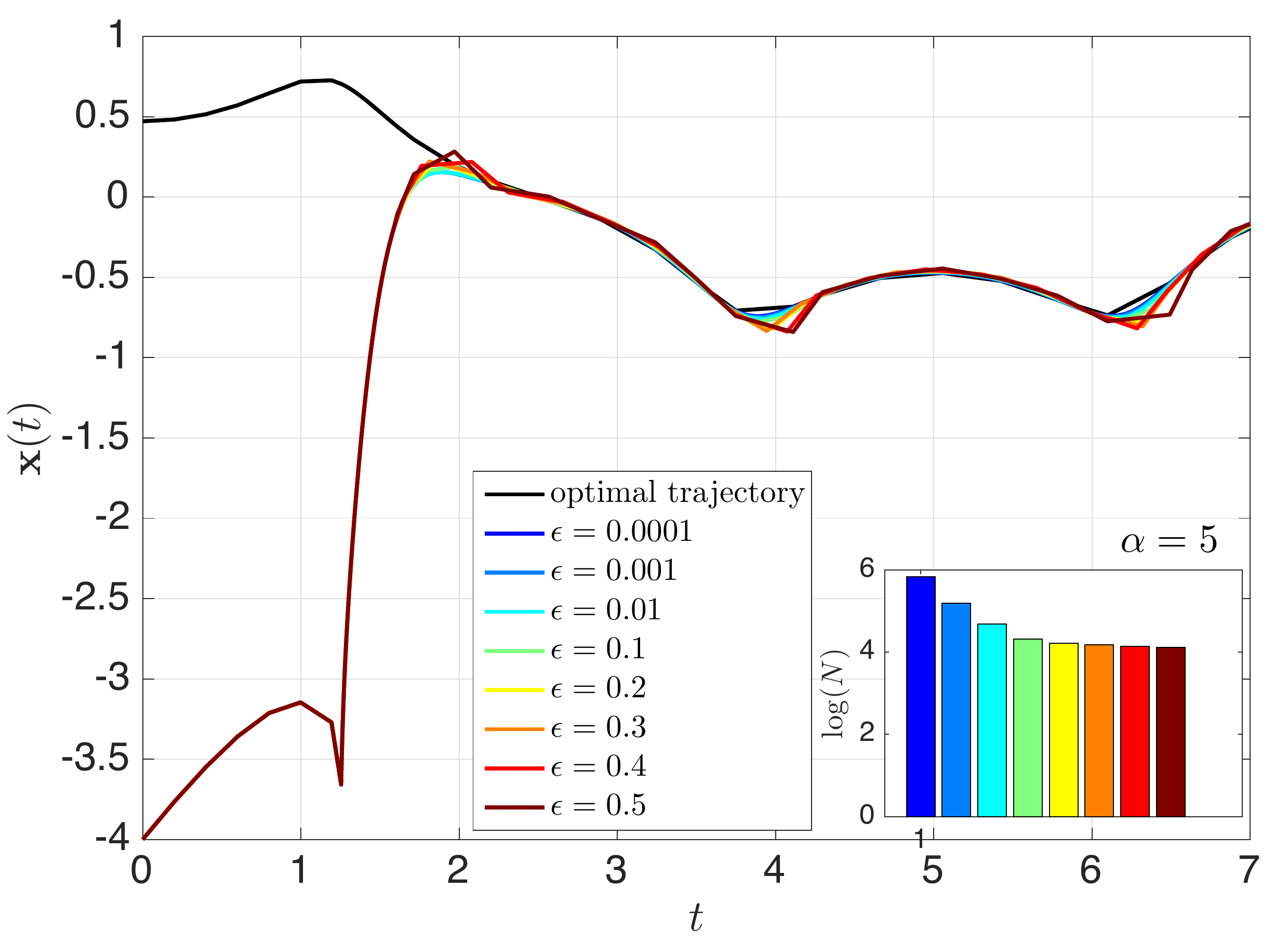}
	\caption{Plot of $\bx(t)$ against $t$ for the self-triggered strategy in Algorithm \ref{tab: Second Order Self-Triggered Optimizer}, and for various values of $\epsilon$.}
	\label{fig: effect of eps}
\end{figure}

\begin{figure}\
	\centering
	\includegraphics[width=0.45\textwidth]{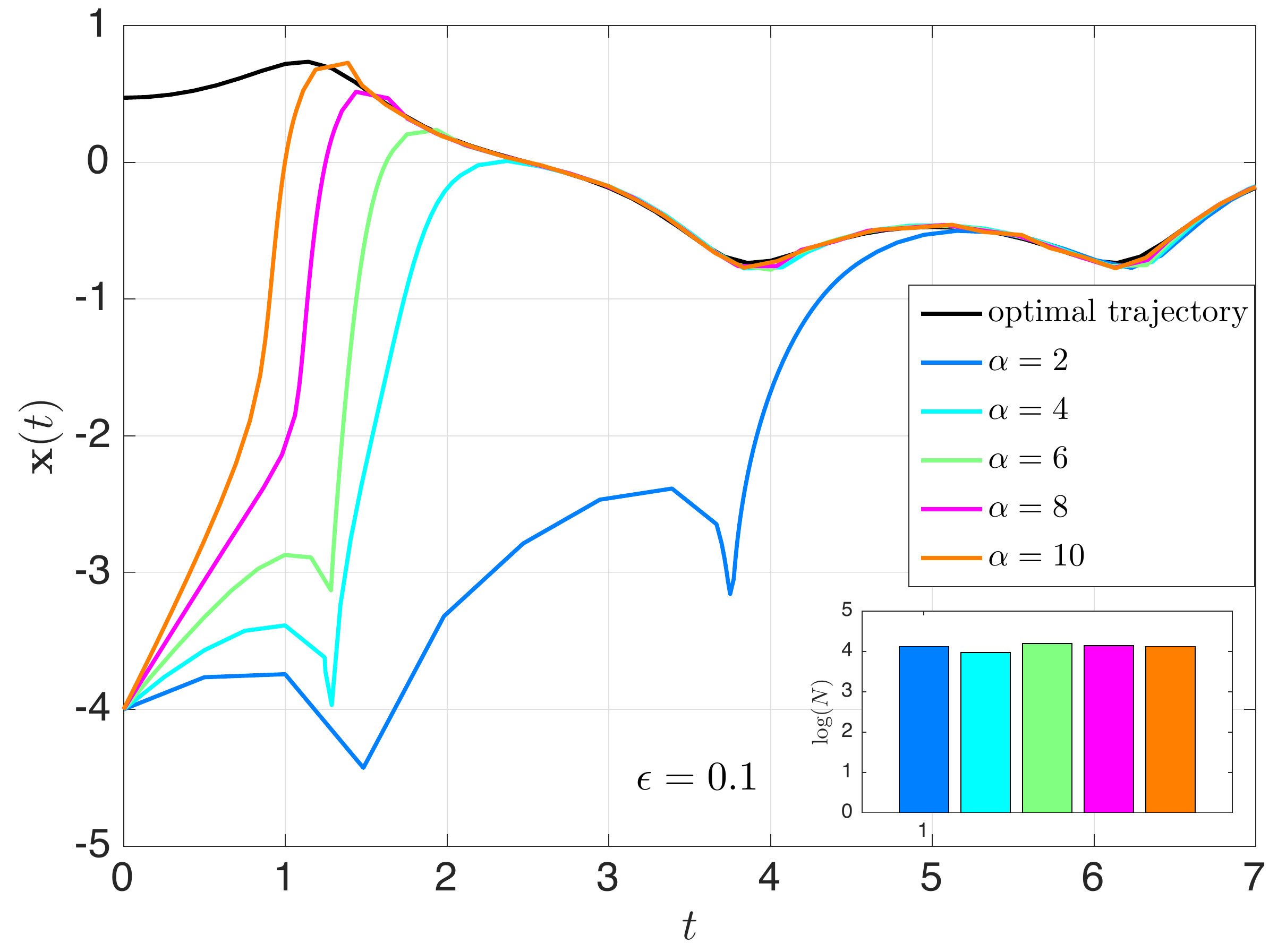}
	\caption{Plot of $\bx(t)$ against $t$ for the self-triggered strategy in Algorithm \ref{tab: Second Order Self-Triggered Optimizer}, and for various values of $\alpha$.}
	\label{fig: effect of alpha}
\end{figure}

\section{Conclusion}

In this paper, we proposed a real-time self-triggered strategy to aperiodically implement a continuous-time dynamics that solves continuously time-varying convex optimization problems. The sampling times are autonomously chosen by the algorithm to ensure asymptotic convergence to the optimal solution while keeping the number of updates at the minimum. We illustrated the effectiveness of the proposed method with numerical simulations. In future work, we will consider the case where the term $\nabla_{\bx t} f(\bx,t)$ in \eqref{eq: descent_vector_field} is not available, and needs to be estimated with backward difference in time.

\appendix
\section{Appendix}

\subsection{Proof of Lemma \ref*{lem: trigerring function}} \label{proof:em: trigerring function}

	We begin by fixing $k \in \mathbb{Z}_{+}$ and analyzing the Lyapunov function during the inter-event time $t \in [t_k,t_{k+1})$. We aim to find a tight upper bound on $\dot{V}(t)$. First, we write $\dot{V}(t)$ in integral form as
	\begin{align}
	\dot{V}(t) = \dot{V}(t_k) + \int_{t_k}^{t} \ddot{V}(\sigma) d\sigma, \ t\geq t_k.
	\end{align}
	Applying Jensen's gives us the inequality
	\begin{align} \label{eq: V_dot integral inequality}
	\dot{V}(t) \leq \dot{V}(t_k) + \int_{t_k}^{t} |\ddot{V}(\sigma)| d\sigma, \ t\geq t_k.
	\end{align}
	The main idea is then to bound $|\ddot{V}(t)|$ on $t\geq t_k$. By adopting the notation $\bg_k(t):=\nabla_{\bx} f_0(\bx(t),t)$, we can rewrite the Lyapunov function as
	\begin{align} \label{eq: Lyapunov function 2}
	V(t) = \dfrac{1}{2} \bg_k(t)^\top \bg_k(t),   \quad  t \geq t_k.
	\end{align}
	By \eqref{eq: discretized dynamics}, we have that $\dot{\bx}(t)=\dot{\bx}(t_k)=\bh(\bx(t_k),t_k)$ for $t\geq t_k$, and therefore, $\bx(t)=\bx(t_k)+\dot{\bx}(t_k)(t-t_k)$ for $t \geq t_k$. Whence, $\bg_k(t)$ reads as
	\begin{align} \label{eq: g_k}
	\bg_k(t)&=\nabla_{\bx} f_0(\bx(t_k)+\dot{\bx}(t_k)(t-t_k),t),\ t \geq t_k.
	\end{align}
	We can write the first two time derivatives of $V(t)$ from \eqref{eq: Lyapunov function 2} as follows,
	\begin{align} \label{eq: V_dot}
	\dot{V}(t) &= \bg_k(t)^\top \dot{\bg}_k(t), &\ t \geq t_k, \nonumber \\
	\ddot{V}(t) &= \dot{\bg}_k(t)^\top \dot{\bg}_k(t) + \bg_k(t)^\top \ddot{\bg}_k(t), &t \geq t_k.
	\end{align}
	In order to bound $\ddot{V}(t)$, we proceed to bound $\bg_k(t)$, $\dot{\bg}_k(t)$, and $\ddot{\bg}_k(t)$, using the known upper bounds granted by Assumption \ref{bounds on second-order derivatives}. We use chain rule to derive $\dot{\bg}_k(t)$ from \eqref{eq: g_k} as follows,
	\begin{align} \label{eq: g_dot}
	\dot{\bg}_k(t) = \nabla_{\bx\bx} f_0(\bx(t),t) \dot{\bx}(t_k) + \nabla_{\bx t} f_0(\bx(t),t),\ t \geq t_k.
	\end{align}
	We now use Assumption \ref{bounds on second-order derivatives}-\emph{(i)} to find an upper bound on $\|\dot{\bg}_k(t)\|_2$ as follows, 
	\begin{align} \label{eq: upperbound g_dot}
	\|\dot{\bg}_k(t)\|_2 &\leq \|\nabla_{\bx\bx} f_0(\bx(t),t)\|_2\| \dot{\bx}(t_k)\|_2 + \| \nabla_{\bx t} f_0(\bx(t),t)\|_2, \nonumber \\
	& \leq C_{xx} \| \dot{\bx}(t_k)\|_2 + C_{xt}, \nonumber \\ 
	& = a_k.
	\end{align}
	where we have used the definition of $a_k$ in \eqref{eq: 2nd order coefficients}. We apply the chain rule again on \eqref{eq: g_dot} to get 
	\ifx Notice that $\dot{\bg}_k(t_k)=-\alpha \nabla_{\bx} f_0(\bx(t_k),t_k)$. Therefore, we get $\dot{V}(t_k) =\bg_k(t_k)^\top \dot{\bg}_k(t_k) = -\alpha \|\nabla_{\bx} f_0(\bx(t_k),t_k)\|_2^2<0$. \fi 
	\begin{align}
	\ddot{\bg}_k(t) = &\left(\sum_{i=1}^{n} \nabla_{\bx\bx x_i} f_0(\bx(t),t)\dot{\bx}_i(t)\right) \dot{\bx}(t_k)\nonumber +\nabla_{\bx tt} f_0(\bx(t),t) \\  &+ 2\nabla_{\bx \bx t} f_0(\bx(t),t)\dot{\bx}(t_k).
	\end{align}
	We use Assumption \eqref{bounds on second-order derivatives}-\emph{(ii)} to bound $\ddot{\bg}_k(t)$. The first term in $\ddot{\bg}_k(t)$ can be bounded as follows,
	\begin{align*}
	& \| \sum_{i=1}^{n} \nabla_{\bx\bx \bx_i} f_0(\bx(t),t)\dot{\bx}_i(t_k) \|_2 \\ &\leq \sum_{i=1}^{n}  \| \nabla_{\bx\bx \bx_i} f_0(\bx(t),t)\dot{\bx}_i(t_k) \|_2 \\
	& \leq  \sum_{i=1}^{n}  \| \nabla_{\bx\bx \bx_i} f_0(\bx(t),t)\|_2 |\dot{\bx}_i(t_k)| \\ & \leq \sum_{i=1}^{n}  C_{xxx} |\dot{\bx}_i(t_k)| \\ &= C_{xxx} \|\dot{\bx}(t_k)\|_1.
	\end{align*}
	The second and third term in $\dot{\bg}_k(t)$ can also be bounded as follows,
	\begin{align*}
	&\|2\nabla_{\bx \bx t} f_0(\bx(t),t)\dot{\bx}(t_k)+\nabla_{\bx tt} f_0(\bx(t),t)\| \\ & \leq 
	2 \|\nabla_{\bx \bx t} f_0(\bx(t),t)\|_2 \|\dot{\bx}(t_k)\|_2 + \|\nabla_{\bx tt} f_0(\bx(t),t)\|_2, \\
	& \leq 2C_{xxt} \|\dot{\bx}(t_k)\|_2 + C_{xtt}.
	\end{align*}
	Putting the last two bounds together, we obtain
	\begin{align} \label{eq: upperbound g_ddot}
	\|\ddot{\bg}_k(t)\|_2 &\leq  \left(C_{xxx} \|\dot{\bx}(t_k)\|_1 + 2C_{xxt}\right) \|\dot{\bx}(t_k)\|_2+ C_{xtt}, \nonumber \\
	&=b_k.
	\end{align}
	where we have used the definition of $b_k$ in \eqref{eq: 2nd order coefficients}.  To bound $\|\bg_k(t)\|_2$, we use Taylor's theorem to expand $\bg_k(t)$ as
	\begin{align*}
	\bg_k(t) = \bg_k(t_k) + \dot{\bg}_k(\xi) (t-t_k),\ t \geq t_k.
	\end{align*}
	for some $t_k \leq \xi < t$. Using the bound $\|\dot{\bg}_k(t)\| \leq a_k$ from \eqref{eq: upperbound g_dot}, we can bound the last result as follows,
	\begin{align} \label{eq: upperbound g_k}
	\|\bg_k(t)\|_2 &\leq \|\bg_k(t_k) \|_2 + \|\dot{\bg}_k(\xi)\|_2 (t-t_k), \nonumber \\
	& \leq \sqrt{2V(t_k)} + a_k (t-t_k).
	\end{align}
	where by \eqref{eq: Lyapunov function 2}, $\|\bg_k(t_k) \|_2=\sqrt{2V(t_k)}$. We substitute all the upper bounds in \eqref{eq: upperbound g_dot}, \eqref{eq: upperbound g_ddot}, and \eqref{eq: upperbound g_k} back in $\ddot{V}(t)$ to write 
	\begin{align} \label{eq: upperbound V_ddot}
	|\ddot{V}(t)| &= \|\dot{\bg}_k(t)^\top \dot{\bg}_k(t) + \bg_k(t)^\top \ddot{\bg}_k(t)\|_2 \\
	& \leq \|\dot{\bg}_k(t)\|_2^2+\|\bg_k(t)\|_2 \|\ddot{\bg}_k(t)\|_2 \nonumber \\
	& \leq a_k^2 + b_k \sqrt{2V(t_k)} + a_k b_k (t-t_k).\nonumber 
	\end{align}
	Plugging the bound \eqref{eq: upperbound V_ddot} back in \eqref{eq: V_dot integral inequality}, evaluating the integral, and using the definition of $\phi_k(t)$ in \eqref{eq: second order triggering function} would let us to conclude that
	\begin{align}
	\dot{V}(t) \leq \phi_k(t),\ t\geq t_k.
	\end{align}
	The proof is complete. \oprocend

	\subsection{Proof of Lemma \ref*{lem: trigerring function third order}} \label{proof:lem: trigerring function third order}
	
	We follow the same logic as  the proof of Lemma \ref*{lem: trigerring function} to find a tight upper bound for $\dot{V}(t)$ by bounding $\ddot{V}(t)$. Notice that if we do not know the bounds in Assumption \ref{bounds on second-order derivatives}-\emph{(i)}, the bound $\|\ddot{\bg}_k(t)\|_2 \leq b_k$ in \eqref{eq: upperbound g_ddot} still holds, whereas the bounds for $\| \dot{\bg}_k(t)\|_2$ in \eqref{eq: upperbound g_dot} and $\|\dot{\bg}_k(t)\|_2$ in \eqref{eq: upperbound g_k} break. To find legitimate bounds for the latter functions, we use Taylor's theorem to express $\dot{\bg}_k(t)$ as follows,
		\begin{align} \label{eq: g_dot mean value}
		\dot{\bg}_k(t) = \dot{\bg}_k(t_k) + \ddot{\bg}_k(\eta)(t-t_k).
		\end{align}
		for some $t_k < \eta < t$. By \eqref{eq: upperbound g_ddot} we know that $\|\ddot{\bg}_k(t)\|_2 < b_k$ for $t \geq t_k$. Hence, we can bound $\|\dot{\bg}_k(t)\|_2$ as
		\begin{align} \label{eq: upperbound g_dot 2}
		\|\dot{\bg}_k(t)\|_2 &\leq \|\dot{\bg}_k(t_k)\|_2 + b_k(t-t_k). \nonumber
		\end{align}
		We use Taylor's theorem one more time to express $\bg_k(t)$ as 
		\begin{align}
		\bg_k(t) = \bg_k(t_k) + \dot{\bg}_k(t_k)(t-t_k) + \frac{1}{2} \ddot{\bg}_k(\xi) (t-t_k)^2.
		\end{align}
		for some $t_k < \xi < t$. Therefore, we can bound $\|\bg_k(t_k)\|_2$ as follows,
		\begin{align}
		\|\bg_k(t)\|_2 \leq \|\bg_k(t_k)\|_2 + \|\dot{\bg}_k(t_k)\|_2(t-t_k) + \frac{1}{2} b_k (t-t_k)^2.
		\end{align}
		We use the obtained bounds for $\|\bg_k(t)\|_2$ and $\|\dot{\bg}_k(t)\|_2$ to bound $|\ddot{V}(t)|$ as follows,
		\begin{align*}
		|\ddot{V}(t)| &= \|\dot{\bg}_k(t)^\top \dot{\bg}_k(t) + \bg_k(t)^\top \ddot{\bg}_k(t)\|_2 \\
		& \leq \left(\|\dot{\bg}_k(t_k)\|_2 + b_k(t-t_k)\right)^2 \\ &+ (\|\bg_k(t_k)\|_2 + \|\dot{\bg}_k(t_k)\|_2(t-t_k) + \frac{1}{2} b_k (t-t_k)^2)b_k. \nonumber
		\end{align*}
		\ifx
		We define $\dot{u}_k(t):=\alpha \sqrt{2V(t_k)}+ b_k(t-t_k)$ to express the last inequality in compact form as $\|\dot{\bg}_k(t)\|_2 \leq \dot{u}_k(t)$. We now bound $\|\bg_k(t)\|_2$ by integrating  \eqref{eq: g_dot mean value}, using Jensen's inequality, and substituting the bound \eqref{eq: upperbound g_dot 2} as follows,
		\begin{align}
		\|{\bg}_k(t)\|_2 &\leq \|{\bg}_k(t_k)\|_2 + \int_{t_k}^{t} \|\ddot{\bg}_k(\sigma)\|_2 d \sigma \nonumber \\
		&\leq \|\bg_k(t_k)\|_2 + \int_{t_k}^{t} \dot{u}_k(\sigma) d \sigma \nonumber \\
		&= \sqrt{2V(t_k)}+ u_k(t)-u_k(t_k).
		\end{align} 
		We substitute all the bounds back in $\ddot{V}(t)$ to write 
		\begin{align} \label{eq: upperbound V_ddot third order}
		|\ddot{V}(t)| &= \|\dot{\bg}_k(t)^\top \dot{\bg}_k(t) + \bg_k(t)^\top \ddot{\bg}_k(t)\|_2 \\
		& \leq \|\dot{\bg}_k(t)\|_2^2+\|\bg_k(t)\|_2 \|\ddot{\bg}_k(t)\|_2 \nonumber \\
		& \leq \dot{u}_k^2(t)+b_k (\sqrt{2V(t_k)}+ u_k(t)-u_k(t_k)).\nonumber 
		\end{align}
		\fi
		Notice that $\|\bg_k(t_k)\|_2=\sqrt{2V(t_k)}$ and $\|\dot{\bg}(t_k)\|_2=\alpha\sqrt{2V(t_k)}$. Finally, we plug the last bound in \eqref{eq: V_dot integral inequality} and use the definition of $\phi_k(t)$ in \eqref{eq: third order triggering function} to conclude that
		\begin{align}
		\dot{V}(t) \leq \phi_k(t),\ t\geq t_k.
		\end{align}
		The proof is complete. \oprocend
	
	\subsection{Proof of Theorem \ref{thm: convergence of second order}} \label{proof:thm: convergence of second order}

		We saw in the proof of Lemma \ref{lem: trigerring function} that for $t_k \leq t < t_{k+1}$, the dynamics of the Lyapunov function satisfies
		\begin{align*}
		\dot{V}(t) \leq \phi_k(t) < \phi_k(t_{k+1}) = 0, \ t_k \leq t < t_{k+1}.
		\end{align*}
		Moreover, $\phi_k(t)$ is convex on $t_k \leq t \leq t_{k+1}$ with boundary values $\phi_k(t_k)=-2\alpha V(t_k)$ and  $\phi_k(t_{k+1})=0$. Hence, we can write
		\begin{align*}
		\phi_k(t) &\leq \left(1- \frac{t-t_k}{t_{k+1}-t_k}\right) \phi_k(t_k)+\left(\frac{t-t_k}{t_{k+1}-t_k}\right) \phi_k(t_{k+1}) \\
		& = \left(1- \frac{t-t_k}{t_{k+1}-t_k}\right) . (-2\alpha V(t_k)).
		\end{align*}
		Therefore, we get the inequality
		\begin{align*}
		\dot{V}(t) \leq \left(1- \frac{t-t_k}{t_{k+1}-t_k}\right) (-2\alpha V(t_k)).
		\end{align*}
		We integrate the above inequality to obtain
		\begin{align}
		V(t_{k+1}) - V(t_{k}) &\leq -\alpha V(t_k) (t_{k+1}-t_k).
		\end{align}
		Moreover, by Remark \ref{rem: sampling times}, For any $k \in \mathbb{Z}_{+}$, the step size $t_{k+1}-t_k$ is strictly positive unless $V(t_k)=0$. In other words, the right-hand side of the above inequality is strictly negative unless $V(t_k)=0$. Therefore, we must have that $\lim\limits_{k \to \infty} V(t_k)=0$. The proof is complete. \oprocend
		
		\subsection{Proof of Theorem \ref{lem: finite time convergence}} \label{proof:thm: No Zeno}
		We provide the proof when the triggering function is given by \eqref{eq: second order triggering function}. For the other triggering function in \eqref{eq: third order triggering function}, the proof follows the same logic and hence, omitted for the sake of brevity.
		
		We first show that for any $k$ with $V(t_k) \geq \epsilon$, we have that $t_{k+1}-t_{k} > \tau_1(\epsilon)$ for some $\tau_1(\epsilon)>0$ to be determined. When $V(t_k) \geq \epsilon$, we are in the first phase of the Algorithm, where we have that $\dot{V}(t) \leq \phi_k(t) < \phi_k(t_{k+1})=0$. Therefore, the Lyapunov function is bounded, and in particular, $V(t_k) < V(t_0)$. It then follows from the definition of $V(t)=V(\bx(t),t)$ in \eqref{eq: Lyapunov function} that, $\|\nabla_{\bx} f_0(\bx(t_k),t_k)\|$ is bounded. This implies that $\|\dot{\bx}(t_k)\|$ is bounded because we have that 
		\[
		\Scale[0.92]{\dot{\bx}(t_k) = -\nabla_{\bbx\bbx}^{-1} f(\bbx(t_k),t_k)
		\Big[\nabla_{\bbx t} f(\bbx(t_k),t_k) + \alpha\nabla_{\bbx} f(\bbx(t_k),t_k)\Big]}.
		\]
		by \eqref{eq: time_varying_newton}. Therefore, $\|\dot{\bx}(t_k)\|_2$ is bounded as
		\[
		\|\dot{\bx}(t_k)\|_2 \leq \frac{1}{m} 	\Big(C_{x t}+ \alpha \|\nabla_{\bbx} f(\bbx(t_k),t_k)\|_2\Big).
		\]
		where we have used the fact that \emph{(i)} $\nabla_{\bx\bx} f(\bx,t) \geq m \mathbf{I}_n$ for all $\bx \in \mathbb{R}^n$ and $t \in \mathbb{R}_{+}$ (see Assumption \ref{strong convexity}); and \emph{(ii)} $\|\nabla_{\bx t} f(\bx,t)\|_2 \leq C_{xt}$ (see Assumption \ref{bounds on second-order derivatives}). Boundedness of $\|\dot{\bx}(t_k)\|_2$ further implies that  the coefficients of the triggering function ($a_k$ and $b_k$ in \eqref{eq: 2nd order coefficients}) are bounded. Thus, $\sup_{k} a_k < a$ and $\sup_{k} b_k<b$ for some $a,b<\infty$. For any fixed $\tau\geq 0$, $\phi_k(t_k+\tau)$ is increasing in $a_k$ and $b_k$. Therefore, $\phi_k(t_k+\tau)$ given in \eqref{eq: second order triggering function} is bounded by
		\begin{align} \label{ap: phi upper bound}
		\phi_k(t_k+\tau) \leq \dfrac{1}{2} ab \tau^2 + (a^2+b \sqrt{2V(t_0)}) \tau -2\alpha \epsilon.
		\end{align}
		where we used the bounds $\epsilon \leq V(t_k)\leq V(t_0)$, $a_k \leq a$, and $b_k \leq b$. The polynomial in the right-hand side has a strictly positive root. Call it $\tau_1(\epsilon)>0$. Therefore, by the substitution $\tau=\tau_1(\epsilon)$ in \eqref{ap: phi upper bound}, we get
		\[
		\phi_k(t_k+\tau_1(\epsilon)) \leq 0 = \phi_k(t_k + \tau_k).
		\]
		where in the second equality, we have used the fact that, according to \eqref{eq: next sampling time}, $\phi_k(t_{k}+\tau_k)=0$, when $V(t_k) \geq \epsilon$. Since $\phi_k(t_k+\tau)$ is an increasing function of its argument, we conclude from the last inequality that
		\[
		0<\tau_1(\epsilon)\leq \tau_{k},\ \mbox{if} \ V(t_k) \geq \epsilon.
		\]
		Ituitively, during the first phase of the Algorithm, the step sizes are lower-bounded by a positive constant, denoted by $\tau_1(\epsilon) > 0$.
		Next, we consider the second phase of the Algorithm where $0\leq V(t_k) \leq \epsilon$. Notice that, in this case, we can bound $\phi_k(t)$ as
		\[
		\phi_k(t_k+\tau) \leq \dfrac{1}{2} ab \tau^2 + (a^2+b \sqrt{2\epsilon}) \tau -2\alpha V(t_k).
		\]
		By integrating the last inequality, and recalling the definition $\psi_k(t_k+\tau)=V(t_k) + \int_{t_k}^{t_k+\tau} \phi_k(\sigma) d\sigma$ in \eqref{eq: psi function}, we obtain the inequality
		\[
		\psi_k(t_k+\tau) \leq \dfrac{1}{6} ab \tau^3 + \frac{1}{2}(a^2+b \sqrt{2\epsilon}) \tau^2 -2\alpha V(t_k)\tau+V(t_k).
		\]
		The right-hand side is an upper bound on the triggering function $\psi_k(t_k+\tau)$. Viewing this bound as a triggering function, it follows that the step size obtained from this upper bound is a \emph{lower} bound on the actual step size obtained by $\psi_k(t_k+\tau)$. More precisely, denote $\tau_k$ as the value of $\tau$ for which the right-hand side of the inequality above is equal to $\epsilon$, i.e., 
		\begin{align} \label{eq: tau'_k}
		\dfrac{1}{6} ab {\tau'_k}^3 + \frac{1}{2}(a^2+b \sqrt{2\epsilon}) {\tau'_k}^2 -2\alpha V(t_k)\tau'_k+V(t_k)=\epsilon.
		\end{align}
		It then follows that $\tau'_k < \tau_k$ where $\tau_k$ is the actual step size that satisfies $\psi_k(t_k+\tau_k)=\epsilon$. Viewing $\tau'_k$ as a function of $\eta:=V(t_k)$ given by the implicit equation above, we wish to find a lower bound on $\tau'_k(\eta)$ when $0\leq \eta\leq \epsilon$. For this purpose, we differentiate the last equation with respect to $\eta$ to obtain
		\[
		\dfrac{1}{2} ab {\tau'_k}^2 \frac{d \tau'_k}{d\eta} + (a^2+b\sqrt{2\epsilon}) \tau'_k \frac{d \tau'_k}{d\eta}-2\alpha\tau'_k-2\alpha \eta \frac{d \tau'_k}{d\eta}+1=0.
		\]
		By setting $d\tau'_k/d\eta=0$ in the last equation, we get the critical value ${\tau_2}:=(2\alpha)^{-1}$. Next, we evaluate $\tau'_k$ for boundary values $\eta=0$, and $\eta=\epsilon$. For $\eta=0$ we obtain from \eqref{eq: tau'_k} that
		\begin{align*}
		\dfrac{1}{6} ab {\tau'_k}^3 + \frac{1}{2}(a^2+b \sqrt{2\epsilon}) {\tau'_k}^2-2\alpha \epsilon\tau_k=0.
		\end{align*}
		The above polynomial has one zero root (ignored by the Algorithm) and a unique positive root, denoted by $\tau_3(\epsilon)$. Hence, in this case $\tau'_k=\tau_3(\epsilon)>0$. On the other hand, for $\eta=0$, we obtain from \eqref{eq: tau'_k} that
		\[
		\dfrac{1}{6} ab {\tau'_k}^3 + \frac{1}{2}(a^2+b \sqrt{2\epsilon}) {\tau'_k}^2-\epsilon.
		\]
		The above polynomial has also a unique positive root, denoted by $\tau_4(\epsilon)>0$. Therefore, it follows that $$\tau'_k \geq \min \{\tau_2, \ \tau_3(\epsilon),\ \tau_4(\epsilon)\}>0.$$ 
		Finally, recall that $\tau'_k$ is a lower bound on $\tau_k$, the selected step size. Hence,
		$$\tau_k \geq \min \{\tau_2, \ \tau_3(\epsilon),\ \tau_4(\epsilon)\}>0.$$
		This confirms that for the case $0 \leq V(t_k) \leq \epsilon$, the step size is strictly lower bounded by a positive function of $\epsilon$. Hence, the proof is complete. \oprocend

\bibliographystyle{IEEEtran}
\bibliography{Refs}

\end{document}